\documentclass[12pt]{amsart}
\usepackage{amsmath, amsthm, amssymb, calrsfs, wasysym, verbatim, bbm, color, graphics, graphicx,enumerate,cite}
\usepackage[left=2.3cm,top=3.3cm,right=2.3cm,bottom=2.7cm]{geometry}
\usepackage{colonequals} 
\usepackage{scrextend}
\usepackage[all,cmtip]{xy}
\usepackage{tikz}
\usepackage{tikz-cd}
\usetikzlibrary{patterns}

\usepackage[
urlcolor=blue, linktocpage, hyperindex=true, colorlinks=true,
linkcolor=blue, citecolor=blue,
]{hyperref}

\hypersetup{pdfborder = {0 0 0}}
\usepackage[nameinlink]{cleveref}

\newtheorem{theorem}{Theorem}[section]
\newtheorem{lemma}[theorem]{Lemma}

\theoremstyle{definition}
\newtheorem{definition}[theorem]{Definition}
\newtheorem{proposition}[theorem]{Proposition}

\newtheorem{example}[theorem]{Example}
\newtheorem{corollary}[theorem]{Corollary}

\theoremstyle{remark}
\newtheorem{remark}[theorem]{Remark}
\numberwithin{equation}{section}

\newcommand{\C}{\mathbb{C}}
\newcommand{\Z}{\mathbb{Z}}
\newcommand{\N}{\mathbb{N}}
\newcommand{\Q}{\mathbb{Q}}

\newcommand{\Pro}{\mathbb{P}}

\newcommand{\Pic}{\mathrm{Pic}}

\newcommand{\Supp}{\mathrm{Supp}\,}

\newcommand{\Ker}{\mathrm{Ker}}

\newcommand{\Bk}{\mathrm{Bk}}

\newcommand{\rank}{\mathrm{rank}\,}

\newcommand{\Oc}{\mathcal{O}}
\newcommand\co{\colon\thinspace}

\crefname{claim}{Cliam}{Claim}
\crefname{corollary}{Corollary}{Conjecture}
\crefname{theorem}{Theorem}{Theorem}
\crefname{conjecture}{Conjecture}{Conjecture}
\crefname{definition}{Definition}{Definition}
\crefname{equation}{Equation}{Eq.}
\crefname{example}{Example}{Example}
\crefname{lemma}{Lemma}{Lemma}
\crefname{proposition}{Proposition}{Proposition}
\crefname{remark}{Remark}{Remark}

\begin{document}
	
	\title[Log canonical maps of open algberaic surfaces]{On open algebraic surfaces of general type whose log canonical maps composed of a pencil}
	
	\author{Hang Zhao}

	\address{ Hang Zhao\\School of Mathematics and Statistics \\ Yunnan University \\  Kunming, Yunnan 650091, PR China}
	\email{zhaoh@ynu.edu.cn}
	\date{\today}
	\thanks{The author was supported in part by the science and technology innovation fund of Yunnan University: ST20210105.}

	
	
	
	
	
	\maketitle
	
	\begin{abstract}
		Let $(S,D)$ be a minimal log pair of general type with $S$ a smooth projective surface and $D$ a simple normal corssing reduced divisor on $S$. We assume that its log canonial linear system $|K_S+D|$ is composed of a penciel, let $f\co S\to B$ be the fiberation induced by the linear system $|K_S+D|$ and $F$ be a general fiber of $f$. Let $b$ (resp. $g$) be the genus of the base curve $B$ (resp. general fiber $F$) and $k=D\cdot F$ the intersection number. We show that \begin{enumerate}
			\item If $k>0$ and $b\geq 2$ then $2\leq g+k \leq 3$, when $g+k=3$ we have $b=2$ and $h^1(S,K_S+D)=0$.
			\item Suppose $p_a(D)\leq 2(l+q(S))+1-h^{1,1}(S)$ where $l$ is the number of irreducible components of $D$, then we have $g\leq 5$ for $p_g(S,D)\gg 0$. Moreover if $p_g(S)=0$, then we have $g\leq 3$.
		\end{enumerate}

		
	\end{abstract}

	\tableofcontents
	\section{Introduction}
	
	The pluricanonical maps of a surface with non-negative kodaira dimension play a important role in the calssification of algebraic surfaces. Let $S$ be a surface, and let $K_S$ be a canonical divisor. Suppose the $n$-genus $P_n(S)=h^0(S,\Oc_S(nK_S))\geq 2$, we can define a rational map $\phi_n\co S\dashrightarrow\Pro^{P_n(S)-1}$, which is called the \emph{$n$th-pluricanonical map} of $S$. Bombieri \cite{Bo73} showed that the $n$th-pluricanonical maps $\phi_n$ are embeddings when $n\geq 5$ and are birational when $n\geq 3$ except for finite many cases; the bicanonical map $\phi_2$ is birational if $K_S^2\geq 10$ and $p_g(S)\geq 6$ except for $S$ admiting a genus 2 fiberation, in which case $\phi_2$ is generically a double covering.
	
   The image $\Sigma$ of the canonical map $\phi_1\co S\dashrightarrow\Pro^{p_g(S)-1}$ is either a surface or a curve. If $\Sigma$ is a curve, we say that the canonical linear system $|K_S|$ is composed of a pencil. Let $f\co X\to B$ be the fiberation induced by the rational map $\phi_1$, i.e., $\mu\co X\to S$ resolve the base locus of $|K_S|$ and take $f$ to be the Stein factorization of the composition morphism $\phi_1\circ\mu \co X\to\Sigma$.
	Let $g,b$ and $q$ be the genus of a general fiber of $f$, the genus of $B$ and the irregularity $q(S):=h^0(S,\Oc_S)$ of $S$. Beauville \cite{Be79} showed that $2\leq g\leq 5$ if $\chi(S,\Oc_S)\geq 21.$ Later, Xiao \cite{Xiao85} showed that either $b=q=1$ or $b=0,q\leq 2$.
	
	If $\Sigma$ is a surface. Beauville \cite{Be79} showed that either $p_g(\Sigma)=0$ or $p_g(\Sigma)=p_g(S)$ in which case $\Sigma$ is of general type. In the first case, if $\chi(S,\Oc_S)\geq 31$, then $\deg\phi_1\leq 9$; in the second case, if $\chi(S,\Oc_S)\geq 14$, then $\deg\phi_1 \leq 3$. More refiment results can be find in \cite{Xiao86}.

	In this paper we will generalize the above result to open algebraic surfaces, for general references here are \cite{It76,It77}.  Let $S^0$ be an open algebraic surface, by Nagata and Hironaka, we can embed it into a complete algebraic surface $S^0\hookrightarrow S$ such that the complement $D=S\setminus S^0$ is a divisor with simple normal corssing singularities.  It is well known that cohomology groups of many vector bundles constructed using logarithmic forms  with poles algong $D$ depend only on $S^0$. Thus if we want to study properties of $S^0$, it is natural to consider the properties of the divisor $K_S+D$ in the spirit of the theory of birational geometry. Thus it is convenient to consider a log pair $(S,D)$ consists of a smooth projective surface $S$ and a reduced effective divisor $D$ on $S$ with only simple normal crossing singularities, such a pair is also called \emph{log smooth surface}. We will consider a \emph{minimal} log smooth surface $(S,D)$, in the sence of Tsunoda which he called \emph{almost minimal} (see for instance \cite{TFu79,ST83}).
	\begin{definition}\label{def: minimal log surfaces}
		We say that a projective log smooth surface $(S,D)$ is \emph{minimal} if $K_S+D$ has the \emph{Zariski-Fujita decomposition }$K_S+D=(K_S+D^{\sharp})+\Bk(D)$ satisfying the following conditions:
		\begin{enumerate}
			\item $\Bk(D)$ is an effecitve $\Q$-divisor such that $\Bk(D)\leq D$;
			\item $K_S+D^{\sharp}$ is nef;
			\item the intersection matrix $\Bk(D)$ is negative defined;
			\item $(K_S+D^{\sharp})\cdot C=0$ for any prime component $C$ of $\Bk(D)$.
		\end{enumerate}
	\end{definition} 
	
	We can define analogous invariants of a log surface as usual surface: $p_g(S,D):=h^0(S,K_S+D)$ the \emph{log geometric genus}, $q(S,D)=h^0(S,\Omega_S^1(D))$ the \emph{log irregularity}, $\kappa(S,D)=\kappa(K_S+D)$ the \emph{log Kodaira dimension}. We say that a log surface $(S,D)$ is of \emph{log general type} if $\kappa(S,D)=2$ or $(K_S+D^{\sharp})^2>0$. 
	
	The log canonical map $\phi\co S\dashrightarrow \Pro^{p_g(S,D)-1}$ has been study by many authors. Sakai \cite{Sa80} considered the case when $D$ is a stable curve and proved a Noether type inequality and a generalization of Bogomolov-Miyaoka-Yau inequality. Soon afterwards, Tsunoda and De-Qi Zhang \cite{TZ92} generalized  the result to the case when $(S,D)$ is minimal log smooth and of log general type.
	 
	 Our purpose is to consider minimal log surfaces $(S,D)$ of log general type whose log canonical linear system $|K_S+D|$ is composed of a penciel. We can assume, by blowing up at the base points of $|K_S+D|$ if necessary, that $|K_S+D|$ is base point free (see Lemma \ref{lem:  invariants}). Let $\phi\co S\to \Sigma\subset\Pro^{p_g(S,D)-1}$ be the correspongding morphism, then $\Sigma$ is a curve. Let $f\co S\to B$ be the Stein factorization of $\phi\co S\to \Sigma$. We will study the fiberation $f\co S\to B$ induced by the linear system $|K_S+D|$. Let $g$ and $b$ denote the genus of a general fiber $F$ of $f$ and the base curve $B$ respectively. It is of interest to know the possibilities of invariants $g,b$ and the intersection numbers $k=D\cdot F$.
	 The main results of this paper are the following.
	 
	 \begin{theorem}\label{thm: main}
	 	With the notations as above. Then we have
	 	\begin{enumerate}
	 		\item If $k>0$ and $b\geq 2$ then $2\leq g+k \leq 3$, when $g+k=3$ we have $b=2$ and $h^1(S,K_S+D)=0$.
	 		\item Suppose $p_a(D)\leq 2(l+q(S))+1-h^{1,1}(S)$ where $l$ is the number of irreducible components of $D$, then we have $3\leq 2g+k\leq 9$ for $p_g(S,D)\gg 0$. Moreover if $p_g(X)=0$, then we have $2g+k\leq 6$.
	 	\end{enumerate}
	 \end{theorem}
 
	Noteworthy, our examples of $b=0$ demonstrates rather strikingly that $g,k$ can be arbitrarily large, see Example \ref{exm: g,k large}, note that $S$ are rational surfaces and they do not satisfy the inequality $p_a(D)\leq 2(l+q(S))+1-h^{1,1}(S)$ of  Theorem \ref{thm: main}.

	\section{Preliminary}
	 Throughout we work over an algebraically closed field of characteristic $0$. We set up notaion and terminology.
	
	A \emph{fiberation} is a projective surjective morphism with connected fibers between two normal varieties.
	
	We write $D\geq 0$ for an effective $\Q$-divisor on a normal variety $X$. Let $E$ be a prime divisor on $X$, we denote by $\mu_D(E)$ the coefficient of $E$ in $D$.
	
	\subsection{Nef vector bundles} In this subsection we recall some basic results on nef vector bundles that will be used throughtout the paper.
	\begin{definition}
		A vector bundle $\mathcal{E}$ on a variety $X$ is said to be \emph{nef} if the tautological line bundle $\Oc_{\Pro(\mathcal{E})}(1)$ is a nef line bundle on the projective bundle $\Pro_X(\mathcal{E})$ over $X$.
	\end{definition}
	
	\begin{proposition}\label{prop: nef vector  bundless}
		Let $\mathcal{E}$ be a vector bundle on a variety $X$.
		\begin{enumerate}
			\item If $\mathcal{E}$ is nef, then so is any quotient bundle $\mathcal{Q}$ of  $\mathcal{E}$.
			\item Let $f\co X\to Y$ be a finite morphism of varieties. If $\mathcal{E}$ is a nef vector bundle on $Y$, then $\mathcal{E}$ is nef if and only if the pull-back $f^*\mathcal{E}$ is a nef vector bundle on $X$.
			\item Let $\mathcal{E}$ be a vector bundle on a variety $X$, then $\mathcal{E}$ is nef if and only if given any finite morphism $\nu\co C \to X$ form a non-singular projective curve $C$ to $X$, and given any quotient line bundle $\mathcal{L}$ of $\nu^*\mathcal{E}$, one has $\deg\mathcal{L}\geq 0$.
		\end{enumerate}
	\end{proposition}
	
	\begin{corollary}\label{cor: degree of nef bundle}
		Let $\mathcal{E}$ be a vector bundle on  a non-singular projective curve $C$ of genus $g$. If $\mathcal{E}\otimes \omega_C^{-1}$ is nef, then $\deg\mathcal{E}\geq 2(g-1)\rank\mathcal{E}$ and thus $\chi(C,\mathcal{E})\geq (g-1)\rank{\mathcal{E}}$.
	\end{corollary}
	\begin{proof}
		By the Riemann-Roch theorem, $\chi(C,\mathcal{E})=\deg\mathcal{E}-\rank\mathcal{E}(g-1)$.  If $\mathcal{E}\otimes\omega_C^{-1}$ is nef, then we have $\deg\mathcal{E}\geq 2(g-1)\rank\mathcal{E}$ and thus $\chi(C,\mathcal{E})\geq  (g-1)\rank\mathcal{E}$.
	\end{proof}
Let $f\co S\to B$ be a fiberation from a smooth projective surface to a smooth curve. By Fujita's theorem \cite{Fu78}, $f_*\omega_{S/B}$ is a nef vector bundle on $B$. Let $D$ be a reduced divisor on $S$ which has only simple normal crossing singularities. To show that the sheaf $f_*\omega_{S/B}(D)$ is also nef we need the following.
	
	\begin{proposition}\cite{CG17}\label{prop: subvector bundle nef}
		Let $\mathcal{E}$ be a vector bundle on a non-singular projective curve $C$, and $\mathcal{F}$ be a subvector bundle of $\mathcal{E}$ with the same rank of $\mathcal{E}$. If $\mathcal{F}$ is nef, then so is $\mathcal{E}$.
	\end{proposition}

		\subsection{Divisors of surfaces} We review some of the standard facts on divisors of algebraic surfaces. 
		
		A reduced divisor $D=\sum C_j$ on a surface $S$ is also called a curve. The \emph{arithmetic genus} of $D$ is $p_a(D)\colonequals 1-\chi(D,\Oc_D)$. By Riemann-Roch formula, we have
		\[K_S\cdot D+D^2=2p_a(D)-2.\]
			\begin{lemma}\label{lem: arithmetic genus}
			Let $D$ be a reduced divisor on a surface $S$. If we write $D=D_1+D_2$ into sum of two reduced divisors $D_1$ and $D_2$, then \[p_a(D)=p_a(D_ 1)+p_a(D_2)+D_1\cdot D_2-1.\]
			In particular, if we write $D=\sum_{i=1}^rD_i$ into sum of all its irreducible components, then we have
			\[p_a(D)=\sum_{i=1}^rp_a(D_i)+1+\sum_{i<j}D_i\cdot D_j -r.\]
		\end{lemma}
		
	The \emph{weighted dual graph} $\Gamma=\Gamma(D)$ of $D$  is the following: each curve $C_j$ in $D$ corresponds to a vertex $v_j$. Two vertices $v_i$ and $v_j$ are connected by an edge of weight $m$ if the corresponding curves has intersection number $C_i\cdot C_j=m$. Each vertex $v_j$ has a weight $n_j=-C_j^2$. The branching muber, denoted by $\beta(v_j)$ or $\beta_D(C_j)$, of a vertex $v_j$ is the muber of vertices which is connected to $v_j$. 
	
	A reduced divisor $D$ is called \emph{simple normal corssing} (SNC for short) if each irreducible component is non-singular and $D$ has at most node. In this case we will say that the pair $(S,D)$ is \emph{log smooth}. The dual graph $\Gamma$ of a SNC divisor $D$ is a simple graph, i.e., it has no loops (an edge connected one vertex to itself) and multiple edges.

		\begin{remark}\label{rem: arithmetic genus}
		Let $D$ be a reduced SNC divisor on $S$ whose dual graph has $r$ vertices and $l$ edges. By Lemma \ref{lem: arithmetic genus}, its arithmetic genus $p_a(D)$ equals $\sum_{i=1}^rp_a(D_i) + 1+l -r.$ Note that, as $D$ is SNC, the sum  $\sum_{i<j}D_i\cdot D_j$ of intersection numbers equals $l$.
	\end{remark}
	
	A $\Q$-divisor is said to be \emph{nef} (resp. \emph{pseudo-effective}) if $D\cdot C\geq 0$ for any irreducible curve $C$ (resp. $D\cdot P$ for any nef $\Q$-divisor $P$) on $S$.  
		
	For a pseudo-effective $\Q$-divisor $D$ on a smooth projective surface $S$, we have the \emph{Zariski-Fujita decomposition} $D=P+N$ which satisfies and is determined by the following numerical properties
	\begin{enumerate}
			\item $P$ and $N$ are $\Q$-divisors;
			\item $N\geq 0$ and the intersection matrix $(C_i\cdot C_j)$ for the prime components $C_i$ of $N$ is negative defined.
			\item $P$ is nef;
			\item $P\cdot C_i=0$ for every prime componet $C_i$ of $N$.  
		\end{enumerate}
	We call $P$ (resp. $N$) the \emph{nef part} (resp. \emph{negative part}) of $D$. For the proofs we refer the reader to \cite{TFu79}.

	\subsection{Theory of Peeling}\label{sec: Theory of Peeling}  A connected curve $E$ in $D$ is said to be \emph{rational} (resp. \emph{admissible}) if it consists only of rational curves (resp. if there are no $(-1)$-curves in $E$ and the intersection matrix of $E$ is negative defiend). Let $L$ be an admissible rational linear chain in $D$. We write $L=\sum_{i=1}^r C_i$ as the sum of its prime components. Suppose $L$ satisfies the condition: $\beta_D(C_1)=1$ and $\beta_D(C_i)=2$ for $2\leq i\leq r-1$. If $\beta_D(C_r)=1$, we call $L$ a \emph{rod}, i.e., $L$ is a connected component of $D$. If $\beta_D(C_r)=2$ and $C_r\cdot C_{r+1}=1$ for a prime component $C_{r+1}$ other than $C_{r-1}$, we call $L$ a \emph{twig}.  An admissible rational twig is called maximal, if it is not extended to an admissible rational twig with more irreducible components of $D$, i.e., $\beta_D(C_{r+1})\geq 3$. A connected curve $F$ in $D$ is called a \emph{fork} if its dual graph is that of the exceptional curves of a minimal resolution of a log terminal singular point and is not a chain. The prime component $C$ of a linear chain $L$ in $D$ with $\beta_D(C)=1$ is called a \emph{tip}.

	Let $(S,D)$ be a log smooth pair with $K_S+D$ is pseudo-effective. In this subsection, we will consider the Zariski-Fujita decomposition
	\[K_S+D=(K_S+D)^++(K_S+D)^-.\]
	Here and subseqently, $(K_S+D)^+$ (resp. $(K_S+D)^-$) will denote the nef part (resp. negative part) of $K_S+D$.
	
	\begin{lemma}\label{lem: negative curve}
		Let $S$ be a smooth projective surface, and $D$ be a $\Q$-boundary on $S$ such that the log pair $(S,D)$ is log smooth. If $E$ is an irreducible curve on $S$ with $E^2\leq -2$ and $(K_S+D)\cdot E<0$, then $E\subset\Supp(D)$ and $E$ is a smooth rational curve.
	\end{lemma}
	
	\begin{proof}
		By assumption, we have $$K_S\cdot E=2p_a(E)-2-E^2\geq 0$$ and thus $D\cdot E<0$, $E$ is a prime component of $D$. Write $D=aE+D'$ with $0<a\leq 1$ and $E\nsubseteq \Supp(D')$, we have
	\[2p_a(E)-2=(K_S+E)\cdot E\leq (K_S+aE+D')\cdot=(K_S+D)\cdot E<0.\]
	Thus $p_a(E)=0$, so that $E$ is a smooth rational curve.
	\end{proof}

	Let $L$ be an admissible rational linear chain in $D$. We write $L=\sum_{i=1}^r C_i$ as the sum of its prime components. Suppose $\beta_D(C_1)=1$, thus $C_1$ is the tip of $L$. By the adjunction formula, for each $C_j$, we have $$(K_S+D)\cdot C_j=-2+\beta_D(C_j).$$
	
	Note that the intersection matrix $(C_i\cdot C_j)$ is negative defined, the following system of $r$ linear equations in $r$ variables has a unique solution 
	\begin{equation}\label{eq: system of linear equations}
		\sum_{i=1}^r a_i C_i\cdot C_j= (K_S+D)\cdot C_j
	\end{equation}
 Since $(K_S+D)\cdot C_1=-1<0$, it follows that $0<a_i\leq 1$ for all $1\leq i\leq r$. The $\Q$-divisor $\Bk(L)=\sum_{i=1}^ra_iC_i$ is called \emph{the bark} of $L$, and the process of substracting $\Bk(L)$ from $L$ is called \emph{the peeling} of $L$. 
 
 Let $F$ be an admissible rational fork, write $F=\sum_{i=1}^rC_i$ as the sum of its prime components. In the same manner we can see that the system of linear equations \ref{eq: system of linear equations} has a unique solution $\Bk(F)=\sum_{i=1}^ra_iC_i$ with $0<a_i\leq 1$ for all $1\leq i\leq r$, which is also called the Bark of $F$.
 
 Let $\{T_{\lambda}\}$, $\{R_{\mu}\}$ and $\{F_{\nu}\}$ be the sets of all admissible rational maximal twigs, all admissible rational rods and all admissible rational forks, respectively, where no prime coponent of $T_{\lambda}$'s belong to $R_{\mu}$'s and $F_{\nu}$'s. Since $T_{\lambda}$'s, $R_{\mu}$'s and $F_{\nu}$'s are disjoint from each other, we can peel the barks of $T_{\lambda}$'s, $R_{\mu}$'s and $F_{\nu}$'s independently. We set
 \[D^{\sharp}= D- \sum_{\lambda}\Bk(T_{\lambda})-\sum_{\mu}\Bk(R_{\mu})-\sum_{\nu}\Bk(F_{\nu}).\]
 The $\Q$-divisor $\Bk(D)=D-D^{\sharp}$ is called \emph{the bark }of $D$.  By the above construction, we have
 
 \begin{enumerate}
 	\item $D^{\sharp}=D-\Bk(D)\geq 0$;
 	\item write $\Bk(D)=\sum_i a_iC_i$, the intersection matrix $(C_i\cdot C_j)_{1\leq i,j\leq r}$ is negative defined;
 	\item $(K_S+D^{\sharp})\cdot C=0$ for every prime component $C$ of $T_{\lambda}$'s, $R_{\mu}$'s and $F_{\nu}$'s;
 	\item $(K_S+D)^+=(K_S+D^{\sharp})^+$ and $(K_S+D)^-=(K_S+D^{\sharp})^-+\Bk(D)$.
 \end{enumerate}
	
	\begin{lemma}\label{lem: exceptional curve of the first kind}
		Let $(S,D)$ be a log smooth pair of log general type with $D$ reduced. With the notations as above. If $E$ is an irreducible curve on $S$ such that $(K_S+D^{\sharp})\cdot E<0$, then $E$ is an exceptional curve of the first kind.
	\end{lemma}

\begin{proof}
	Consider the Zariski-Fujita decompositon $K_S+D^{\sharp}=P+N$, where $P$ is the nef part and $N$ is the negative part of $K_S+D^{\sharp}$. Since $(S,D)$ is of log general type, we have $P^2>0$. By assumption, $(K_S+D^{\sharp})\cdot E<0$ and note that $P\cdot E\geq 0$, it follows that $N\cdot E<0$ and thus $E\subset\Supp(N)$. By the properties of Zariski-Fujita decompositon, we have $P\cdot E=0$. By the Hodge index theorem, we see that $E^2\leq 0$. Again by the properties of Zariski-Fujita decompositon, the intersection matrix of supports of $N$ is negative defined, we infer that $E^2<0$. It remains to prove that the case $E^2\leq -2$ can not occur.
	
	Suppose $E^2\leq -2$, then by Lemma \ref{lem: negative curve},  we have $E\subset\Supp(D^{\sharp})$ and $E$ is a smooth rational curve. Since $(K_S+D^{\sharp})\cdot C=0$ every prime component $C$ of $\Bk(D)$, it follows that $\mu_{D^{\sharp}}(E)=1$. Let $C_1,\dots,C_k$ be the prime components of $D$ that meets $E$ where $k=\beta_D(E)$. Set $a_i=\mu_{D^{\sharp}}(C_i)$, then we have $(K_S+D^{\sharp})\cdot E= -2+a_1+\cdots +a_k$. Note that $(K_S+D)\cdot E=-2+k$, if $k\leq 1$, then $E$ is either an isolated component of $D$ or a tip of an admissible rational maximal twig or an admissible rational rod, in either case we have $\mu_{\Bk(D)}(E)>0$, which contradicts to $\mu_{D^{\sharp}}(E)=1$. Therefore we can assume that $k\geq 2$. By assumption, we have $a_1+\cdots+a_k<2$ and $k\geq 2$. Thus at least one $i$ with $a_i<1$, say $a_1<1$. We can conclude that there exists an admissible maximal rational twig $T$ of $D$ with $C_1\leq T$ and $E\nsubseteq\Supp(T)$. If $k=2$, then $T+E$ is also an admissible rational twig of $D$, which contradicts the maximality of $T$. By Lemma \cite[Lemma 6.16]{Fu82}, $a_i=1-\frac{1}{d_i}$ with $d_i\geq 2$ or $a_i=1$, let $l$ be the number of $i$ with $a_i<1$, so we have $k-2<\frac{1}{d_1}+\cdots +\frac{1}{d_l}$ and $k\geq 3$. It follows that $l=k=3$. Therefore $a_i<1$ for all $1\leq i\leq 3$, then there exists an admissible rational fork $F$ such that $E+C_1+C_2+C_3\leq F$, in this case $\mu_{\Bk(F)}(E)>0$, which contradicts to $\mu_{\Bk(D)}(E)=0$ and the proof is complete.
 \end{proof}

	After  a succesion of several contractions of excetional curve $E$ as in Lemma \ref{lem: exceptional curve of the first kind}, we obtain log smooth pair $(\bar{S},\bar{D})$ and a birational morphism $\sigma\co S\to \bar{S}$ where $\bar{D}=\sigma_*(D)$, which satisfies the following properties:
	\begin{enumerate}
		\item $K_{\bar{S}}+\bar{D}^{\sharp}$ is nef, thus the Zariski-Fujita decomposition of $K_{\bar{S}}+\bar{D}$ is $K_{\bar{S}}+\bar{D}^{\sharp}+\Bk(\bar{D})$, where $(K_{\bar{S}}+\bar{D})^+=K_{\bar{S}}+\bar{D}^{\sharp}$ and $(K_{\bar{S}}+\bar{D})^-=\Bk(\bar{D})$.
		\item $(K_S+D)^+=\sigma^*(K_{\bar{S}}+\bar{D}^{\sharp})$.
	\end{enumerate}

We say a log smooth pair $(S,D)$ with a reduced boundary $D$ and $\kappa(S,D)\geq 0$ \emph{minimal} if $K_S+D^{\sharp}$ is nef. Given a log smooth pair $(S,D)$ with a reduced boundary $D$ and $\kappa(S,D)=2$, we can run $(K_S+D^{\sharp})$-MMP and terminete with a minimal model $(\bar{S},\bar{D})$ which is minimal in the sence of Tsunoda, which is called \emph{almost minimal}. For more details we refer the reader to \cite[Theorem 1.3]{ST83}.

\begin{lemma}\label{lem: bound of Bk^2}
	Let $(S,D)$ be a log smooth pari with a reduced boundary $D$. Let $t$ be the number of tips of $\Supp(\Bk(D))$. Then we have 
	$\Bk(D)^2\geq -t$.
\end{lemma}

\begin{proof}
		Let $L$ be an admissible rational linear chain in $D$. We write $L=\sum_{i=1}^r C_i$ as the sum of its prime components such that $C_i\cdot C_{i+1}=1$ for $1\leq i\leq r-1$ and $\beta_D(C_1)=1$. Write $\Bk(L)=\sum_{i=1}^ra_iC_i$, then we have $$\Bk(L)^2=\sum_{i=1}^r a_i\Bk(L)\cdot C_i=\sum_{i=1}^r a_i(K_S+D)\cdot C_i=\sum_{i=1}^r a_i(-2+\beta_D(C_i)).$$
		So if $L$ is a rod, $\Bk(L)^2=-a_1$; if $L$ is a twig, $\Bk(L)^2=-a_1-a_r$. Let $F$ be a fork of $D$, write $F=C_0+T_1+T_2+T_3$. Let $C_1,C_2,C_3$ be the tip of $T_1,T_2,T_3$ respectively. Set $a_i=\mu_{\Bk(F)}(C_i)$ for $i=0,1,2,3$. Similar arguments apply to $F$, we have $\Bk(F)^2=a_0-a_1-a_2-a_3$. Note that in all cases, $0<a_i\leq 1$. Thus we have $\Bk(D)^2\geq -t$.
\end{proof}

	\subsection{The Bogomolov-Miyaoka-Yau inequality to log surfaces}
	Let $(S,D)$ be a log smooth pair with a smooth projective surface $S$ and a reduced divisor $D$ on $S$. The sheaf of logarithmic 1-forms  $\Omega_S^1(\log D)$ is locally free of rank $2$. Its chern classes $\bar{c}_1=c_1(\Omega_S^1(\log D)),\bar{c}_2=c_2(\Omega_S^1(\log D))$ are defined, and we regard $c_1(\Omega^1_S(\log D))=K_S+D$ as a class in $H^2(S,\Z)$.
	
	\begin{proposition}\label{prop: Noether equality}
		Let $(S,D)$ be a log smooth pair. Then
		\[\bar{c}_2= e(S)+2(p_a(D)-1-l),\]
		where $e(S)$ is the topological Euler characteristic of $S$ and $l$ is the number of edges of the dual graph of $D$ . Moreover, we get the log version of Noether's equality
		\[\bar{c}_1^2+\bar{c}_2+6(p_a(D)-1)+D^2+2l=12\bar{\chi},\]
		where $\bar{\chi}=\chi(S,K_S+D)$.
	\end{proposition}
	
	\begin{proof}
		Write $D=\sum_{i=1}^r D_i$, wherre $D_i$'s are the irreducible components of $D$. For each $1\leq i\leq r$, put $W_i=\sum_{j=1}^i D_j$ thus $W_r=D$ and $W_1=D$, and we define $W_0=0$, then there are exact sequences
		\begin{equation}
			0\to \Omega_S^1(\log W_{i-1})\to \Omega_S^1(\log W_i)\to \Oc_{D_i}(W_{i-1}|_{D_i})\to 0.
		\end{equation}
		From this we see 
		\begin{align*}
			c_2(\Omega_S^1(\log W_{i}))&=c_2(\Omega_S^1(\log W_{i-1}))+(K_S+W_{i-1})\cdot D_i+D_i^2-W_{i-1}\cdot D_i\\
			&= c_2(\Omega_S^1(\log W_{i-1}))+2(p_a(D_i)-1).
		\end{align*}
		Summing up these equalities for $1\leq i\leq r$, we have
		\[c_2(\Omega_S^1(\log D))=c_2(\Omega_S^1)+2\sum_{i=1}^r(p_a(D_i)-1).\]
		Note that $c_2(\Omega_S^1)=e(S)$, and $\sum_{i=1}^r(p_a(D_i)-1)=p_a(D)-1-l$ (see Remark \ref{rem: arithmetic genus}).
	\end{proof}
	
	\begin{proposition}\label{prop: topological Euler characteristic}
		Let $e(S-D)$ be the topological Euler characteristic of the open algebraic surface $S-D$, then we have $e(S-D)=c_2(\Omega_S^1(\log D))$.
	\end{proposition}
	
	\begin{proof}
		By a theorem of Deligne (cf. \cite[Corollary 3.2.13 and Corollary 3.2.14]{De71}), the spectral sequence 
		\[E_1^{p,q}=H^q(X,\Omega^p_S(\log D))\implies H^{p+q}(S-D,\C)\]
		degenerates at $E_1$ page, thus we have Hodge decomposition
		\[H^n(S-D,\C)\cong H^{p,q}(S,\Omega_S^p(\log D)).\]
		From this, we can see that 
		\[e(S-D)=\chi(S,\Oc_S)-\chi(S,\Omega_S^1(\log D))+\chi(S,\Omega_S^2(\log D)).\]
		The Riemann-Roch theorem gives
		\begin{align*}
			\chi(S,\Omega_S^1(\log D))&=\frac{1}{2}(K_S\cdot D+D^2)-c_2(\Omega_S^1(\log D))+2\chi(S,\Oc_S)\\
			&= \chi(S,\Omega_S^2(\log D)) -c_2(\Omega_S^1(\log D)) +\chi(S,\Oc_S).
		\end{align*}
		Therefore, $c_2(\Omega_S^1(\log D))=e(S-D)$.
	\end{proof}

	To show Theorem \ref{thm: Main 1}, we need the following.
	\begin{proposition}\label{prop: e leq pg+1}
		Let $(S,D)$ be a log smooth pair with a reduced boundary $D$. Suppose $p_a(D)\leq 2(l+q(S))+1-h^{1,1}(S)$, where $l$ is the number of edges of the dual graph of $D$. Then we have $$e(S-D)\leq 2p_g(S,D)+1.$$
		If $p_g(S)=0$, we have $e(S-D)\leq p_g(S,D)+1$.
	\end{proposition}
	
	\begin{proof}
		By Proposition \ref{prop: Noether equality} and Proposition \ref{prop: topological Euler characteristic}, we have 
		\begin{equation}\label{eq: chern 1}
		e(S-D)=e(S)+2p_a(D)-2-2l.
		\end{equation} 
		Note that \begin{equation}\label{eq: chern 2}
		e(S)=2-4q(S)+2p_g(S)+h^{1,1}(S).
		\end{equation}
		Substituting (\ref{eq: chern 2}) into (\ref{eq: chern 1}) yields
		\begin{equation}\label{eq: chern 3}
	e(S-D)=h^{1,1}(S)+2p_g(S)-4q(S)+2p_a(D)-2l.
		\end{equation} 
		 On the other hand, the Riemann-Roch theorem gives
		\begin{equation}\label{eq: chern 4}
		\chi(S,\Omega_S^1(\log D))=p_a(D)-1-e(S-D)+2\chi(S,\Oc_S).
		\end{equation}
		Substituting (\ref{eq: chern 3}) into (\ref{eq: chern 4}) we obtain
		\begin{equation}
	\chi(S,\Omega_S^1(\log D))=2(q(X)+l)+1-h^{1,1}(X)-p_a(D).
		\end{equation}
	By assumption we have $\chi(S,\Omega_S^1(\log D))\geq 0$. By the proof of Proposition \ref{prop: topological Euler characteristic}, we have
		\[e(S-D)=\chi(\Oc_S)+\chi(K_S+D)-\chi(\Omega_S^1(\log D))\leq 1+p_g(S)+p_g(S,D).\]
		Note that $p_g(S)\leq p_g(S,D)$, which proves the proposition.
	\end{proof}
	
	\begin{theorem}\cite[Theorem  1.1]{Miy84}
		Let $(S,D)$ be a log smooth projective surface with a reduced boundary $D$ such that $K_S+D$ is pseudo-effective. Let $K_S+D=P+N$ be the Zariski-Fujita decomposition, where $P$ is the nef part and $N$ is the negative part. Then the inequality 
		\[\frac{1}{3}P^2\leq c_2(\Omega_S^1(\log D))-\frac{1}{4}N^2\]
		holds.
	\end{theorem}
	
	\begin{corollary}\label{thm: BMY inequality}
		Let $(X,D)$ be a minimal log smooth surface of log general type. Then we have 
		\[\frac{1}{3}(K_X+D^{\sharp})^2\leq c_2(\Omega_X^1(\log D))-\frac{1}{4}\Bk(D)^2.\]
	\end{corollary}

\subsection{Linear systems on rational surfaces}

Consider a rational surface $S$ for which we have a sequence of blowing ups
\begin{equation}\label{eq: blow ups}
\pi\co S=S_n\xrightarrow{\pi_n} S_{n-1}\xrightarrow{\pi_{n-1}} \cdots \xrightarrow{\pi_2} S_1\xrightarrow{\pi_1} S_0=\Pro^2,
\end{equation}

	where $\pi_i\co S_i\to S_{i-1}$ is the blowing-up of $S_{i-1}$ at a point $p_i$ with exceptional divisor $E_i$ on $S_i$ for $i=1,\dots,n$. Let $\bar{E}_i$ be the total transfomational of $E_i$ on $S$, and $H$ be the class of total transform of a line in $\Pro^2$, then $H,\bar{E}_1,\dots, \bar{E}_n$ forms a free basis of $\Pic(S)$. The canonical class $K_S$ of $S$ is given in terms of $H,\bar{E}_1,\dots, \bar{E}_n$ as $-3H+\bar{E}_1+\cdots+\bar{E}_n$.
	
	Let $p\in \Pro^2$ be a point. We say that $q$ is a \emph{infinitely near point }to $P$ of \emph{order} $n$ on $\Pro^2$, if there exists a birational morphism as (\ref{eq: blow ups}), such that $p_1=p$, $\pi_i(p_{i+1})=p_i$, $i=1,\dots,n$ and $q\in E_n$. We denote by $\mathfrak{B}(\Pro^2)$ the set of infinitely near points on $\Pro^2$. A \emph{cluster} $\mathcal{K}$ on $\Pro^2$  is a finite subset of $\mathfrak{B}(\Pro^2)$.
	
	Given a cluster $\mathcal{K}=\{p_1,\dots,p_n\}$ on $\Pro^2$ , we denote by $\pi_{\mathcal{K}}\co S_{\mathcal{K}}\to\Pro^2$ the blowing-up of all the points in $\mathcal{K}$. For a curve on $\Pro^2$, we denote by $\bar{C}^{\mathcal{K}}$ (or simply $\bar{C}$ if $\mathcal{K}$ is understood) the total transform of $C$, and $\widetilde{C}^{\mathcal{K}}$ (or simply $\widetilde{C}$ if $\mathcal{K}$ is understood)  the strict transform of $C$. We define the \emph{multiplicity} of $C$ at the point $p_i\in \mathcal{K}$, denote by $m_i(C)$, as the multiplicity of the strict transform $\widetilde{C}$ at $p$ on $S_{\mathcal{K}}$, we have 
	\[\bar{C}^{\mathcal{K}}=\widetilde{C}^{\mathcal{K}}+\sum_{i=1}^nm_i(C)\bar{E}_i,\]
	where $E_p$ is the exceptional divisor over $P$.
	
 	Given an effective divisor $E=\sum_{i=1}^n\nu_i \bar{E}_i$ on $S_{\mathcal{K}}$. The ideal sheaf $$\mathcal{I}=\pi_{\mathcal{K}*}\Oc_{S_{\mathcal{K}}}(-E)$$
 	defines a zero-dimensional subscheme $Z$ of $\Pro^2$. For any $d\in\N$, $H^0(\Pro^2,\Oc_{\Pro^2}(d)\otimes\mathcal{I})$ defines a sublinear system $V$ of $H^0(\Pro^2,\Oc_{\Pro^2}(d))$ with base ideal $\mathcal{I}$.
 	We see that $C\in V$ if and only if $m_i(C)\geq \nu_i$ for all $1\leq i\leq n$. Thus the linear system $V$ is called a linear system of plane curves of degree $d$ with assigned base points $p_1,\dots,p_n$ with respective multiplicities at least $\nu_1,\dots,\nu_n$, and we say it is type $(d;\nu_1,\dots,\nu_n)$.  
 	
 	\begin{lemma}\label{lem: big p2}
 		Let $\mathcal{K}=\{p_1,\dots,p_n\}$ be a cluster on $\Pro^2$, and $V$ a linear system of plane curves of degree $d$ with assigned base points $\mathcal{K}$ of type $(d;\nu_1,\dots,\nu_n)$.  We have
 		\begin{enumerate}
 			\item $\dim V \geq \frac{1}{2}d(d+3)-\frac{1}{2}\sum_{i=1}^n \nu_i(\nu_i+1) .$
 			\item The divisor $dH-\nu_1\bar{E}_1-\cdots -\nu_n\bar{E}_n$ is big if $d^2-\sum_{i=1}^n \nu_i^2>0.$
 		\end{enumerate}
 		
 	\end{lemma}
 	
 	\begin{proof}
 		For (1), see \cite[Section 1.2]{ACM02}. For (2), since $$h^0(S_{\mathcal{K}},m(dH-\sum_{i=1}^nv_i\bar{E}_i))\geq \frac{1}{2}dm(dm+3)-\frac{1}{2}\sum_{i=1}^nv_im(v_im+1)=\frac{1}{2}(d^2-\sum_{i=1}^nv_i^2)m^2+O(m)$$
 		it follows that if $d^2-\sum_{i=1}^nv_i^2>0$ then $dH-\nu_1\bar{E}_1-\cdots -\nu_n\bar{E}_n$ is big.
 		\end{proof}
 	
 	Consider the Hirzebruch surface $\Sigma_e=\Pro(\Oc\oplus\Oc(e))$ of degree $e\geq 0$. Let $\pi\co \Sigma_e\to\Pro^1$ be the $\Pro^1$-fiberation, let $\Delta_0$ and $\Gamma$ be its section with $\Delta_0^2=-e$ and its general fiber. Set $\Delta_{\infty}=\Delta_0+e\Gamma$, then $\Delta_{\infty}$ is the class of tautological line bundle of $\Sigma_e$ and we have $\Delta_{\infty}^2=e$. 
 	
 	\begin{lemma}\label{lem:big Hri}
 		Let $\mathcal{K}=\{p_1,\dots,p_n\}$ be a cluster on $\Sigma_e$, consider the linear system $V$ on $\Sigma_e$ with assigned base points $\mathcal{K}$ with repsective multiplicities $\nu_1,\dots,\nu_n$. Consider the very ample divisor $a\Delta_{\infty}+b\Gamma$ where $a$ and $b$ are positive integers.
 		Let $\mu\co S\to \Sigma_e$ be the blowing-up at $\mathcal{K}$. We have
 		\begin{enumerate}
 			\item $	\dim V\geq (a+1)(b+\frac{ae}{2})+a-\sum_{i=1}^n\frac{1}{2}\nu_i(\nu_i+1).$
 			\item The divisor $\mu^*(a\Delta_{\infty}+b\Gamma)-\sum_{i=1}^n\nu_i\bar{E}_i$ is big if
 			$a(b+\frac{ae}{2})-\frac{1}{2}\sum_{i=1}^n\nu_i^2>0.$
 		\end{enumerate}
 	\end{lemma}
 
 \begin{proof}
  For $a\geq 0,b\in\Z$ , we have 
 \[\pi_*\Oc(a\Delta_{\infty} +b\Gamma)=\mathrm{Sym}^a(\Oc\oplus\Oc(e))\otimes\Oc(b)=\sum_{i=0}^a\Oc(b+ie)\]
 Therefore $h^0(\Sigma_e,a\Delta_{\infty}+b)= (a+1)(b+\frac{ae}{2}+1)$ and (1) holds. For (2), since
 \begin{align*}
 h^0(S,m(\mu^*(a\Delta_{\infty}+b\Gamma)-\sum_{i=1}^n\nu_i\bar{E}_i))&\geq (am+1)(bm+
\frac{ame}{2})-\frac{1}{2}\sum_{i=1}^nv_im(v_im+1)\\
&=(a(b+\frac{ae}{2})-\frac{1}{2}\sum_{i=1}^n\nu_i^2)m^2+O(m)
 \end{align*}
it follows that if $a(b+\frac{ae}{2})-\frac{1}{2}\sum_{i=1}^n\nu_i^2>0$ then $\mu^*(a\Delta_{\infty}+b\Gamma)-\sum_{i=1}^n\nu_i\bar{E}_i$ is big.
 \end{proof}
 	
	\subsection{Semi-positivity of $f_*\omega_{S/B}(D)$}
	Let $f\co S\to B$ be a fiberation from a smooth projective surface to a smooth curve.
	Let $D$ be a reduced divisor on $S$ which has only simple normal corssing singularities. In this section we will show that $f_*\omega_{S/B}(D)$ is a nef vector bundle on $B$. We need the following positivity result of a fiberation under some monodromy conditions and some assumption on the divisor $D$, which was proved by Fujino \cite{OFu04}.
	\begin{theorem}\label{thm:  log semi-positivity}
		Let $f\co X\to Y$ be a surjective morphism from a non-singular projective variety $X$ to a projective variety $Y$. Let $D=\sum D_i$ be a simple normal crossing divisor on $X$. Assume the following conditions hold
		\begin{enumerate}
			\item $D$ is strongly horizontal with respect to $f$, i.e., each irreducible component of $D_{i_1}\cap D_{i_2}\cap\dots\cap D_{i_k}$ is dominant onto $Y$. 
			\item   Let $\Sigma$ be a simple normal crossing divisor on $Y$ such that $f$ is smooth  and $D$ is relatively normal crossing over over $Y_0=Y\setminus\Sigma$.
			\item Put $X_0=f^{-1}(Y_0),D_0=D\cap X_0,f_0=f|_{X_0}$ and $d=\dim X-\dim Y$. All the local monodromies on the local system $R^{d+i}f_{0*}\C_{X_0\setminus D_0}$ around every irreducible component of $\Sigma$ are unipotent.
		\end{enumerate} Then $R^if_*\omega_{X/Y}(D)$ is a nef vector bundle on $Y$.
	\end{theorem}
	
	\begin{remark}
		In this paper we only need the Theorem \ref{thm:  log semi-positivity} for the case when $X$ is a surface and $Y$ is a curve, for this case the above theorem is a consequence of the work of Steenbrink and Zucker on variation of mxied Hodge structures (see \cite[Section 5]{SZ85}). 
		The condition (1) becomes that $D$ is a disjoint union of non-singular curves $C_j$ which are horizontal with respect to $f$. The condition (2) is automatically satisfied. The condition (3) can be achieved after replacing $Y$ by a finite convering and taking a suitable birational model of the base change of $X$.
	\end{remark}
	
	\begin{lemma}\label{lem: strongly horizontal modify}
		Let $f\co X\to Y$ be a fiberation of non-singular surface $X$ to a curve $Y$ and $D$ be a SNC divisor on $X$. Then there exists an birational morphism $g\co Z\to X$ from a non-singular surface $Z$ to $X$ such that the horizontal part of $D'$ is strongly horizontal with respect to $h=f\circ g$ where $D'$ is the strict transform of $D$ and $g^*D$ is also SNC. Moreover, if $h_*\omega_{Z/Y}(D')$ is nef, then so is $f_*\omega_{X/Y}(D)$.
	\end{lemma}
	
	\begin{proof}For simplicity we may assume that $D$ is horizontal with respect to $f$.
		If $D$ is not strongly horizontal with respect to $f$, then there exist two horizontal components $D_1$ and $D_2$ of $D$ meet at one point $P$. There are finitely many of points which are the intersection of two horizontal components of $D$. Let $g\co Z\to X$ be the blowing up at these points, it is clearly that the strict transform $D'$ is strictly horizontal.
		
		Note that $g_*\omega_Z(g^*D)=\omega_X(D)$, so that $h_*\omega_{Z/Y}(g^*D)=f_*\omega_{X/Y}(D)$. Since $g^*D-D'$ is effective and its support is vertical with respect to $h$, we have an inclusion $h_*\omega_{Z/Y}(D')\hookrightarrow h_*\omega_{Z/Y}(g^*D)$. Note that $\rank h_*\omega_{Z/Y}(g^*D)=h^0(F,K_F+g^*D|_F)=h^0(F,K_F+D'|_F)=\rank h_*\omega_{Z/Y}(D')$, where $F$ is a general fiber of $h$. By Proposition \ref{prop: subvector bundle nef}, it follows that $f_*\omega_{X/Y}(D)=h_*\omega_{Z/Y}(g^*D)$ is nef.
	\end{proof}
	
	\begin{theorem}
		Let $f\co S\to B$ be a fiberation from a smooth surface $S$ to a non-singular curve $B$ and Let $D$ be a SNC-divisor on $S$. Then $f_*\omega_{S/B}(D)$ is a nef vector bundle.
	\end{theorem}
	
	\begin{proof}
		By Lemma \ref{lem: strongly horizontal modify}, we may assume $D$ is strongly horizontal. By Kawamata's uniponent reduction \cite[Corollary 18]{Ka81}, we have the following commutative diagram
		$$\begin{tikzcd}
			S' \ar[rrr, bend left,"\nu"] \ar[r,"\nu_2"] \ar[d,"f'"'] &S^{\nu} \ar[r,"\nu_1"] \ar[d,"f^{\nu}"']&	S_{B'} \ar[d,"p"'] \ar[r,"\nu_0"] & S \ar[d,"f"] \\
			B'\ar[r,equal]  &B'\ar[r,equal]&	B' \ar[r,"\mu"] & B 
		\end{tikzcd}$$
		where $\mu$ is a finite morphism of curves, $\nu_1$ is the normalization of the fiber product $S_{B'}:= S\times_B B'$ and $\nu_2$ is a desingularization of $S^{\nu}$, such that the induced fiberation $f'\co S'\to B'$ satisfies the conditions (1) through (3) of Theorem \ref{thm:  log semi-positivity}. Since $\mu$ is flat, the relative dualizing sheaves compatible with base change $\omega_{S_{B'}/B'}=\nu_0^*\omega_{S/B}$. Consider the reduced divisor $D'=(\nu_0^*D)_{\mathrm{red}}$ on $S_{B'}$,  we have an injective homomorphism $$p_*\omega_{S_{B'}/B'}(D')\hookrightarrow p_*\nu_0^*\omega_{S/B}(D),$$ because $D'\leq \nu_0^*D$.  By the base change theorem, $p_*\nu_0^*\omega_{S/B}(D)\cong \mu^*f_*\omega_{S/B}(D)$, and so we get a injection
		\[p_*\omega_{S_{B'}/B'}(D')\hookrightarrow\mu^*f_*\omega_{S/B}(D).\]
		
		Since $\Supp (\nu_0^*D-D') $ is contained in the ramification locus of $\nu_0$ and is vertical with respect to $p$, it follows that the cokernal of the homomorphism $p_*\omega_{S_{B'}/B'}(D')\hookrightarrow\mu^*f_*\omega_{S/B}(D)$ is a torsion sheaf on $B$, thus we can regard $p_*\omega_{S_{B'}/B'}(D')$ as a subvector bundle of $\mu^*f_*\omega_{S/B}(D)$ with the same rank. By Proposition \ref{prop: subvector bundle nef} and  Proposition \ref{prop: nef vector  bundless} (2), to show that $f_*\omega_{S/B}(D)$ is nef it is sufficitent to show that $p_*\omega_{S_{B'}/B'}(D')$ is nef.
		Note that the relative dualizing sheaf $\omega_{S^{\nu}/S_{B'}}$ is determined by $$\nu_{1*}\omega_{S^{\nu}/S_{B'}}=\mathcal{H}om_{S_{B'}}({\nu}_{1*}\Oc_{S^{\nu}},\Oc_{S_{B'}})$$
		
		therefore $\nu_{1*}\omega_{S^{\nu}/B'}=\mathcal{H}om_{S_{B'}}(\nu_{1*}\Oc_{S^{\nu}},\omega_{S_{B'}/B'})$, and hence we get an injective homomorphism $\alpha\co \nu_{1*}\omega_{S^{\nu}/B'}\hookrightarrow \omega_{S_{B'}/B'}$. Set $D^{\nu}=(\nu_1^*D')_{\mathrm{red}}$, push down by $p_*$, we get injections 
		\[f^{\nu}_*\omega_{S^{\nu}/B'}(D^{\nu})\hookrightarrow f^{\nu}_*\omega_{S^{\nu}/B'}(\nu_1^*D')\hookrightarrow p_*\omega_{S_{B'}/B'}(D').\]
		For the same reason, the cokernal of $f^{\nu}_*\omega_{S^{\nu}/B'}(D^{\nu})\hookrightarrow p_*\omega_{S_{B'}/B'}(D')$ is a torsion sheaf on $B'$, so it is sufficient to show that $f^{\nu}_*\omega_{S^{\nu}/B'}(D^{\nu})$ is nef. Set $D''= (\nu_2^* D^{\nu})_{\mathrm{red}}$, we wish to choose $\nu_2$ as Lemma \ref{lem: strongly horizontal modify} so that the horizontal part of $D''$ is strongly horizontal with respect to $f'$.  By the construction we have $\nu_{2*}\omega_{S'}(D'')=\omega_{S^{\nu}}(D^{\nu})$,  thus
		\begin{align*}
		f'_*\omega_{S'/B'}(D'')&= f^{\nu}_*\nu_{2*}\omega_{S'}(D'')\otimes\omega_{B'}^{-1}\\
		&= f^{\nu}_*\omega_{S^{\nu}}(D^{\nu})\otimes\omega_{B'}^{-1}\\
		&=f^{\nu}_*\omega_{S^{\nu}/B'}(D^{\nu}).
		\end{align*} 	By Theorem \ref{thm:  log semi-positivity}, $f'_*\omega_{S'/B'}(D'')$ is nef, this complete the proof.
	\end{proof}

	\section{Proofs of the Main Theorem}
	

We begin with some properties of fiberation of algebraic surface.	Let $f\co X\to B$ be a fiberation of a non-singular projective surface $X$ onto a non-singular projective curve $B$ of genus $b$. Let $g$ be the genus of a general fiber $F$ of $f$ and $q=h^1(X,\Oc_X)$ be the irregularity of $X$.  Let $D$ be a reduced effective divisor on $X$.
	
	\begin{lemma}\label{lem: h^1(K_X+Ds)}Assume $D\cdot F>0$ where $F$ is a general fiber of $f$, then we have 
		\[h^1(B,f_*\Oc_X(K_X+D))=h^1(X,K_X+D).\]
	\end{lemma}
	\begin{proof}
		By the Leray spectral sequence, we have \[h^1(X,K_X+D)=h^1(B,f_*\Oc_X(K_X+D))+h^0(B,R^1f_*\Oc_X(K_X+D)).\]
		
		By the relative duality, \[\mathcal{H}om_{\Oc_B}(R^1f_*\Oc_X(K_X+D),\Oc_B)\cong f_*\Oc_X(-D)\otimes\omega_B^{-1}.\] It is sufficient to show that  $R^1f_*\Oc_X(K_X+D)=0.$
		By assumption $D\cdot F>0$ for a general fiber $F$ of $f$. We conclude that $R^1f_*\Oc_X(-D)=0$ by base change of cohomology, i.e. $(f_*\Oc_X(-D))_x\cong  H^0(F,\Oc_F(-D|_F))=0$ where $x$ is a point on $B$ and $F$ is the fiber over $x$. Therefore, we have $R^1f_*\Oc_X(K_X+D)=0.$
	\end{proof}

	Let $(S,D)$ be a minimal log surface of log general type. Assume $p_g(S,D)\geq 2$ and the log canonical linear system $|K_S+D|$ is composed of a pencil, let $\phi\co S\to \Sigma=\phi(S)\hookrightarrow\Pro^{p_g(S,D)-1}$ be the rational map defined by $|K_S+D|$, where $\Sigma$ is a curve. We will study the algebraic fiberation $f\co X\to B$ which is obtained by the following diagram:
	$$\begin{tikzcd}
		X \ar[r,"\mu"]  \ar[d,"f"'] \ar[dr,"\psi"]  & S \ar[d,dashed,"\phi"] &  \\
		B \ar[r,"\tau"] & \Sigma \ar[r,hook] & \Pro^{p_g(S,D)-1}
	\end{tikzcd}
	$$
	where $\mu$ is a log resolution which resolve the base points of the linear system $|K_S+D|$ such that the movable part of  $\mu^*|K_S+D|$ is base point free, it induces a morphism $\psi\co X\to \Sigma=\psi(X)$. We can write $K_X+\mu_*^{-1}D=\mu^*(K_S+D)+E_+-E_-$ where $E_+$ and $E_-$ are effective $\mu$-exceptional reduced divisors having no commen component. The fiberation $f$ is the Stein factorization of the morphism $\psi.$ Put $D_X=\mu_*^{-1}D+E_-$, then  $(X,D_X)$ is agian log smooth. For a log surface $(S,D)$, we are interesting their invariants $(p_g(S,D),h^1(S,K_S+D),p_a(D))$.
	According the following lemma \ref{lem:  invariants}. After replacing $(S,D)$ by $(X,D_X)$, we may assume the movable part of $|K_S+D|$ is base point free.
	The irregularity $q(S)$ of a normal surface $S$ with rational singularities is a birational invariant. For log surface $(S,D)$, this is ture for $h^1(S,K_S+D)$.
	
	\begin{lemma}\label{lem:  invariants}
		Let $(S,D)$ be a projective log smooth log surface.
		For any log resolution $\mu\co Y\to S$ of singularities of $(S,D)$, we can write $K_Y+\mu_*^{-1}D+E_-=\mu^*(K_S+D)+E_+$. Put $D_Y=\mu_*^{-1}D+E_-$.
		Then we have
		\begin{enumerate}
			\item  $h^i(Y,K_Y+D_Y)=h^i(S,K_S+D)$ for $i=0,1$;
			\item $p_a(D_Y)=p_a(D)$;
			\item $\chi(Y,K_Y+D_Y)=\chi(S,K_S+D).$
		\end{enumerate}
		
	\end{lemma} 
	
	\begin{proof}
		For (1), it is clear when $i=0$, we show that the equality holds for $i=1$. By projection formula, we have \[R^1\mu_*\Oc_Y(K_Y+D_Y)=R^1\mu_*\Oc_S(\mu^*(K_S+D)+E_+)=\Oc_S(K_S+D)\otimes R^1\mu_*\Oc_Y(E_+).\]
		It is sufficitent to show that $R^1\mu_*\Oc_Y(E_+)=0$.
		Since $S$ is non-singular, we have $R^1\mu_*\Oc_Y=0$, we can assume that $E_+\neq 0$. 
		Note that $D_Y\cdot E_+=0$, then we get $K_Y\cdot E_+=E_+^2<0$. Consider the exact sequence
		\[\begin{tikzcd}
			0\ar[r ] &\Oc_Y \ar[r] & \Oc_Y(E_+) \ar[r] & \Oc_{E_+}(E_+) \ar[r] & 0
		\end{tikzcd}\]
		It follows that $h^1(E_+,\Oc_{E_+}(E_+))=\chi(Y,\Oc_Y)-\chi(Y,\Oc_Y(E_+))=-\frac{1}{2}E_+(E_+-K_Y)=0$. Push down the above exact sequence by $\mu_*$, we botain an exact sequence 
		\[\begin{tikzcd}
			R^1\mu_*\Oc_Y \ar[r] & R^1\mu_*\Oc_Y(E_+) \ar[r] & H^1(E_+,\Oc_{E_+}(E_+)) 
		\end{tikzcd}\]
		So we have $R^1\mu_*\Oc_Y(E_+)=0$.
		
		For (2), \begin{align*}
			p_a(D_Y)&=\frac{1}{2}D_Y\cdot(K_Y+D_Y)+1\\
			&=\frac{1}{2}D_Y\cdot(\mu^*(K_S+D)+E_+)+1\\
			&=\frac{1}{2}D_Y\cdot(\mu^*(K_S+D))+1\\
			&=\frac{1}{2}D\cdot(K_S+D)+1=p_a(D)
		\end{align*}
		The second equality uses $D_Y\cdot E_+=0$ and the fourth equality uses $\mu_*D_Y=D$.
		
		For (3), we may assume $D\neq 0$, since $h^2(Y,K_Y+D_Y)=h^0(Y,-D_Y)=0$ (resp. $h^2(S,K_S+D)=h^0(S,-D)=0$). By (1), we have 
		\[\chi(Y,K_Y+D_Y)=h^0(Y,K_Y+D_Y)-h^1(Y,K_Y+D_Y)=h^0(S,K_S+D)-h^1(S,K_S+D)=\chi(S,K_S+D)\]
	\end{proof}
	Now we may consider a log smooth surface $(S,D)$ satisfying the following conditions
	\begin{enumerate}
		\item $(S,D)$ is minimal in the sense of the Definition \ref{def: minimal log surfaces}; 
		\item the adjoint linear system $|K_S+D|$ is base point free and is composed of a pencil.
	\end{enumerate}
	
	 Let $f\co S\to B$ be the fiberation induced by $|K_S+D|$. Then we can write $K_S+D\equiv nF+Z$, where $F$ is a general fiber of $f$, $Z$ is the fixed part of $|K_S+D|$, and $n=\deg(B/\Sigma)\deg\Sigma$ a positive integer. Since $\deg\Sigma\geq p_g(S,D)-1$, we have $n\geq p_g(S,D)-1$. Let $K_S+D=K_S+D^{\sharp}+\Bk(D)$ be the Fujita-Zariski decomposition as Section \ref{sec: Theory of Peeling}, we need the following inequality which is a part of Noether's inequality of log surfaces of general type which was proved by Tsunoda and De-Qi Zhang \cite{TZ92}.

	\begin{lemma}\label{2g-2+k leq P2}
		$g\leq \frac{n+2}{2n^2}(K_X+D^{\sharp})^2+1$.
	\end{lemma}
	\begin{proof}
		Write $D^{\sharp}=\sum_{i=1}^r \alpha_i D_i$ where $D_i$'s are the irreducible components and $0<\alpha_i\leq 1$. Suppose $\alpha_i<1$ for $1\leq i\leq t$. Note that  the intersection maxtrix $(D_i\cdot D_j)_{1\leq i,j\leq t}$ is negative defined. So we can define $\Q$-divisors $F^{\sharp}$ and $Z^{\sharp}$ as follows:
		\begin{align*}
			&F^{\sharp}=F+\sum_{i=1}^t a_i D_i\text{, with }F^{\sharp}\cdot D_j=0\text{ for all }1\leq j\leq t\\
			&Z^{\sharp}=Z+\sum_{i=1}^t b_i D_i\text{, with }Z^{\sharp}\cdot D_j=0\text{ for all }1\leq j\leq t.\\
		\end{align*}
	By \cite[Lemma 1.1]{TZ92}, we have
		\begin{enumerate}
			\item $K_S+D^{\sharp}\equiv nF^{\sharp}+Z^{\sharp}$;
			\item $F\leq F^{\sharp}$ and $0\leq Z^{\sharp}\leq Z$;
			\item $(F^{\sharp})^2\geq 0$ and $(K_S+D^{\sharp})^2\geq n(K_S+D^{\sharp})\cdot F^{\sharp}\geq n^2(F^{\sharp})^2$.
		\end{enumerate}
		Let $C_i$ ($1\leq i\leq p$) (resp. $E_j$ ($1\leq j\leq q$)) be the components of $D^{\sharp}$ meeting $F$ with coefficients $\alpha_i<1$ (resp. $\alpha_i=1$). Then we have 
		\begin{enumerate}
			\item $(K_S+D^{\sharp})\cdot F^{\sharp}\geq K_S\cdot F+\sum_{i=1}^p (1-\frac{2}{\gamma_i})C_i\cdot F+\sum_{j=1}^q E_j\cdot F$, where $\gamma_i=-C_i^2$;
			\item $(F^{\sharp})^2\geq \sum_{i=1}^p\frac{(C_i\cdot F)^2}{\gamma_i}$;
			\item $(K_S+D^{\sharp})\cdot F^{\sharp}+2(F^{\sharp})^2\geq K_S\cdot F+\sum_{i=1}^pC_i\cdot F+\sum_{j=1}^qE_j\cdot F\geq 1$.
		\end{enumerate}
		This implies that $(\frac{n+2}{n^2})(K_S+D^{\sharp})^2\geq (K_S+D^{\sharp})\cdot F\geq 2g-2$, which proves the lemma.
		
	\end{proof}

	\begin{theorem}\label{thm: Main 1}
		Let $(S,D)$ be a minimal log smooth surface of log general type with $p_g(S,D)\geq 2$. Assume $p_a(D)\leq 2(l+q(S))+1-h^{1,1}(S)$. Suppose  the log canonical linear system $|K_S+D|$ is composed of a penciel. Let $f\co S\to B$ be the induced fiberation to a smooth curve $B$ whose general fiber has genus $g$. Then $g\leq 5,$ for $p_g(S,D)\gg 0$. Moreover if $p_g(S)=0$, then $g\leq 3$. 
	\end{theorem}
	
	\begin{proof}
		Let $K_S+D=P+N$ be the Zariski-Fujita decomposition. By Theorem \ref{thm: BMY inequality}, we have $\frac{1}{3}P^2+\frac{1}{4}N^2\leq c_2(\Omega_X^1(\log D))$.  By Proposition \ref{prop: topological Euler characteristic} and Proposition \ref{prop: e leq pg+1}, we have
		\begin{equation}\label{eq: Main Thm 1}
		\frac{1}{3}P^2+\frac{1}{4}N^2\leq 2p_g(S,D)+1\leq 2n+3
		\end{equation}
		
		 Substituting Lemma \ref{2g-2+k leq P2} into \ref{eq: Main Thm 1}, we get 
		\[g-1\leq \frac{3(n+2)(2n+3-\frac{1}{4}N^2)}{2n^2}\]
		 By Lemma \ref{lem: bound of Bk^2}, $-t< N^2\leq 0,$ where $t$ is the number of tips of $N$. Therefore, for $n\gg t>0$, we have $\frac{3(n+2)(2n+3-\frac{1}{4}N^2)}{2n^2}\leq 4$, and thus $g\leq 5$.
		
		If $p_g(X)=0,$ then $c_2(\Omega^1_X(\log D))\leq p_g(X,D)+1\leq n+2$, thus the above inequality becomes
		\[g-1\leq\frac{3(n+2)(n+2-\frac{1}{4}N^2)}{2n^2}\leq 2\]
		for $n\gg t$.
	\end{proof}
\begin{remark}Let $(S,D)$ be a log smooth surface as Theorem \ref{thm: Main 1}. By the Riemann-Roch theorem, we have 
$$\chi(S,\omega_S(D))-\chi(S,\Oc_S)=\frac{1}{2}(K_S+D)\cdot D.$$ By Serre duality, $h^2(S,\omega_S(D))=h^0(S,-D)=0$. Note that $p_a(D)=\frac{1}{2}(K_S+D)\cdot D+1$ and $p_g(S,D)=h^0(S,\omega_S(D))$, then we get
\begin{equation}
p_g(S,D)=h^1(S,\omega_S(D))-q(S)+p_g(S)+p_a(D).
\end{equation}
On the other hand, by the Serre duality, $h^1(S,K_S+D)=h^1(S,-D)$. Let $k=\dim\Ker(\varphi) $, where $\varphi$ is the homomorphism  $\varphi\co H^1(S,\Oc_S)\to H^1(D,\Oc_D)$ induced by 
\[ 0\to \Oc_S(-D)\to \Oc_S\to \Oc_D\to 0\] 
So we have $h^1(S,K_S+D)=m+k-1$, where $m$ is the number of connected components of $D$. Therefore, the invariant $h^1(S,K_S+D)$ measures both the irregularity $q(S)$ of $S$ and the number of connected components of $D$. In particular, if $D$ is connected, then $h^1(S,K_S+D)\leq q(S)$. If $h^1(S,K_S+D)=0$, then $D$ is connected. 

	If $S$ is a rational surface and $D$ is connected, then we get $p_g(S,D)=p_a(D)$. Since our proof of Theorem \ref{thm: Main 1} relays on $p_g(S,D)\gg 0$. In this case we require that $p_a(D)\gg 0$ or $p_a(D)< 2l+1$. We will give some examples of rational surfaces which violate this assumption and show that their $g$ have no bound.
\end{remark}

	\begin{theorem} 
		Let $(S,D)$ be a porjective log smooth surface such that the adjoint linear system $|K_S+D|$ is composed of a pencil. Suppose $f\co S\to B$ is the fiberation induced by the linear system $|K_S+D|$, let $g$ and $b$ be the genus of a general fiber $F$ of $f$ and the genus of $B$. Assume $k=D\cdot F>0$. Then the following cases occurs
		\begin{enumerate}
			\item if $b\geq 2$, then $(g,k)=(1,1),(1,2)$ or $(2,1)$;
			\item if $g+k=3$, then $b\leq 2$ and $h^1(S,K_S+D)=0$;
		\end{enumerate}
	\end{theorem}
	
	\begin{proof}
		Let $\mathcal{L}$ be the subinvertible sheaf of $f_*\omega_S(D)$ generically generated by global sections of $f_*\omega_S(D)$. Then we have $h^0(B,\mathcal{L})=h^0(B,f_*\omega_S(D))=p_g(S,D)$.  Consider the exact sequence of locally free sheaves
		\begin{equation}\label{eq: Main thm 2.1}
		\begin{tikzcd}
		0\ar[r ] &\mathcal{L} \ar[r] & f_*\omega_S(D) \ar[r] & \mathcal{Q} \ar[r] & 0
		\end{tikzcd}
		\end{equation}
	 
		Put $k=D\cdot F$. By the base change of cohomology and the Riemann-Roch theorem, we have $\rank f_*\omega_S(D)=h^0(F,K_F+D|_F) = g+k-1$ and thus $\rank \mathcal{Q}= g+k-2$. 
		
		Since $f_*\omega_S(D)\otimes\omega_B^{-1}$ is nef, by Proposition \ref{prop: nef vector  bundless} (1),so $\mathcal{Q}\otimes\omega_B^{-1}$ is also nef. By Corollary \ref{cor: degree of nef bundle} we obtain 
		\begin{equation}\label{eq: Main thm 2.2}
		\deg\mathcal{Q}\geq 2(b-1)(g+k-2)\text{ and }h^0(B,\mathcal{Q})\geq \chi(B,\mathcal{Q})\geq (b-1)(g+k-2).
		\end{equation} Since $h^0(B,\mathcal{L})=h^0(B,f_*\omega_S(D))$, from the log exact sequence induced from (\ref{eq: Main thm 2.1}), we infer that \begin{equation}\label{eq: Main thm 2.3}
		h^0(B,\mathcal{Q})\leq h^1(B,\mathcal{L})=h^0(B,\omega_B\otimes\mathcal{L}^{-1})\leq b-1
		\end{equation} the second inequality follows from the fact $\deg\mathcal{L}>0$. Substituting (\ref{eq: Main thm 2.2}) into (\ref{eq: Main thm 2.3}) yields $$(b-1)(g+k-2)\leq b-1.$$ Thus when $b\geq 2$ we have $g+k\leq 3$. Note that $g+k-1=\rank f_*\omega_S(D)\geq 1$, we have $g+k\in\{2,3\}$.
		
		By Lemma \ref{lem: h^1(K_X+Ds)}, the log exact sequence induced from (\ref{eq: Main thm 2.1}), (\ref{eq: Main thm 2.2}) and (\ref{eq: Main thm 2.3}), we have $$h^1(S,K_S+D)=h^1(B,f_*\omega_S(D))=h^1(B,\mathcal{L})-\chi(B,\mathcal{Q})\leq (b-1)(3-g-k).$$ Therefore, if $b=1$ or $g+k=3$, we have $h^1(S,K_S+D)=0$.
		
		Suppose $b\geq 2$ and $g+k=3$. In this case  $h^1(S,K_S+D)=0$, all equalities of (\ref{eq: Main thm 2.2}) and (\ref{eq: Main thm 2.3}) hold. In particular, we have $\rank f_*\omega_S(D)=2$ and  $h^0(B,\omega_B\otimes\mathcal{L}^{-1})=b-1$. By the Clifford inequality,\begin{equation}
		b-1=h^0(B,\omega_B\otimes\mathcal{L}^{-1})\leq \frac{1}{2}\deg(\omega_B\otimes\mathcal{L}^{-1})+1=b-\frac{1}{2}\deg\mathcal{L},
		\end{equation}  hence $\deg\mathcal{L}\leq 2$. On the other hand, $\deg\mathcal{Q}=2(b-1)(g+k-2)=2(b-1)$. So $\deg f_*\omega_S(D)=\deg\mathcal{L}+\deg\mathcal{Q}\leq 2b$. Since $f_*\omega_S(D)\otimes\omega_B^{-1}$ is nef, we have $2b\geq \deg f_*\omega_S(D)\geq 4(b-1)$. So we get $b=2$. 
	
	\end{proof}

	\section{Examples}
	
	The following example illustrate that when $g=k=1$, the genus $b$ of the base curve $B$ can be any value.
   \begin{example}\label{ex 1}
   		Let $f\co S\to B$ be a relative minimal elliptic fiberation, where $B$ is a non-singular curve of genus $b$. Let $D$ be a section of $S$. By canonical bundle formula $\omega_S=f^*(\omega_B\otimes\mathcal{L}),$ where $\mathcal{L}$ is an invertible sheaf on $B$ such that $\deg \mathcal{L} =\chi(S,\Oc_S)$. Suppose $p_g(S)=\chi(S,\Oc_S)-1+b\geq 2$. 
   	
   	We will show that $D$ is contianed in the fixed part of $|K_S+D|$. Let $F$ be a general fiber of $f$, then $D$ intersects $F$ at a unique point $P=D\cap F$. For any $E\in |K_S+D|$. By Riemann-Roch formula $h^0(F,(K_S+D)|_F)=1$, we have  $|(K_S+D)|_F|=P$, thus $P\in E$. It follows that there exists a dense open set $U\subset B$, such that $f^{-1}(U)\cap D\subset E$. Thus $D\leq E$. Therefore the linear system $|K_S+D|$ is composed of a pencil, and the induced fiberation is $f\co S\to B$, thus we have $g=k=1$.

   \end{example}

The following two examples show that when $b=0$, the intersection number $k$ can be arbitrarily large.
	 
	\begin{example}\label{ex 2}
		Let $C_1,C_2\subset\Pro^2$ be quadratic plane curves such that $C_1\cap C_2=\{p_1,p_2,p_3,p_4\}$. Let $p_5,p_6,p_7$ be three distinct points on $C_1\setminus\{p_1,p_2,p_3,p_4\}$ and $p_8\in C_2\setminus\{p_1,p_2,p_3,p_4\}$. Let $C_3$ be a quadratic curve through $p_5,p_6,p_7,p_8$. Then $C=C_1+C_2+C_3$ be a plane curve of degree $6$ passing through points $p_1,\dots,p_8$ with multiplicities are all $2$. Let $\pi\co S\to \Pro^2$ be the blowing-up at $p_1,\dots,p_8$. Then the strict transform $\widetilde{C}$ belongs to $|6H-2E_1-\cdots -2E_8|$, take $D=\widetilde{C}$. Note that $K_S+D\sim 3H-E_1-\cdots -E_8.$ By Lemma \ref{lem: big p2}, $K_S+D$ is big. Note that $\widetilde{C}_1\sim 2H-E_1-\cdots -E_7$ and $(K_S+D)\cdot \widetilde{C}_1=-1<0,$ thus $\widetilde{C}_1$ is contained in the fixed part of $|K_S+D|$. Since $K_S+D-(2H-E_1-\cdots -E_7)= H-E_8$ and $(H-E_8)^2=0$, the linear system $|K_S+D|$ is composed of a pencil. Let $\pi'\co S'\to \Pro^2$ be the blowing up at $p_8$, the pencil of lines passing through $p_8$ induces a fiberation $f'\co S'\to\Pro^1$ whose general fiber is $\Pro^1$. By construction, the birational morphism $\pi\co S\to\Pro^2$ factors through $\pi'\co S'\to \Pro^2$, we have a birational morphism $\mu\co S\to S'$. The composition $f=\mu\circ f'\co S\to \Pro^1$ is a fiberation which is induced by $|K_S+D|$, thus we have $g=b=0$. Note that the general fiber $F$ of $f$ is numerically equivalent to $H-E_8$, thus $k=D\cdot F=4$.
		
	\end{example}

\begin{example}\label{ex 3}
	Let $a\geq 2$ be an integer, and $\mathcal{K}=\{p_0,p_1,\dots,p_{4a-4}\}$ be a cluster on $\Pro^2$, where $p_0,p_1,\dots,p_{4a-4}$ are points on $\Pro^2$. Consider the linear system of plane curves of degree $3a$ with assigned base points $\mathcal{K}$ of type $(3a;3a-3,2^{4a-4})$. Let $C\subset\Pro^2$ be an irreducible member  of this linear system. Let $\pi\co S\to \Pro^2$ be the blowing-up at $\mathcal{K}$. Let $E_i$ be the exceptional divisor over $p_i$, $\widetilde{C}$ be the strict transform of $C$ on $S$ and $H$ be the total transform of the class of line on $\Pro^2$, we have \[\widetilde{C}\sim 3aH-(3a-3)E_0-2\sum_{i=1}^{4a-4}E_i.\]
	Take $D=\widetilde{C}$, consider the log surface $(S,D)$, we have \[ K_S+D\sim (3a-3)H-(3a-4)E_0-\sum_{i=1}^{4a-4}E_i.\] By Lemma \ref{lem: big p2}, $K_S+D$ is big and the linear system $|(a-1)H-(a-2)E_0-\sum_{i=1}^{4a-4}E_i|$ is non-empty. Take $G\in |(a-1)H-(a-2)E_0-\sum_{i=1}^{4a-4}E_i|$ we have \[(K_S+D)\cdot G<0,\]
	thus $G$ is contained in the fixed part of $|K_S+D|$. Sustracting $(a-1)H-(a-2)E_0-\sum_{i=1}^{4a-4}E_i$ from $K_S+D$, we get $(2a-2)(H-E_0)$, which is a pencil. For the same reason as Example \ref{ex 2}, the induced fiberation $f\co S\to \Pro^1$ has invariants $g=b=0$ and $k=3a$.
\end{example}

The following example show that if we drop the assumption $p_a(D)\leq 2(l+q(S))+1-h^{1,1}(S)$ of Theorem \ref{thm: Main 1}, then the invariants $k$ and $g$ can be arbitrarily large.

\begin{example}\label{exm: g,k large}
	Let $0\leq e\leq g$ be integers, consider the Hirzebruch surface $\Sigma_e=\Pro(\Oc\oplus\Oc(e))$ of degree $e$. Let $\pi\co \Sigma_e\to\Pro^1$ be the $\Pro^1$-fiberation, let $\Delta_0$ and $\Gamma$ be its section with $\Delta_0^2=-e$ and its general fiber. Set $\Delta_{\infty}=\Delta_0+e\Gamma$, then $\Delta_{\infty}$ is the class of tautological line bundle of $\Sigma_e$ and we have $\Delta_{\infty}^2=e$. Set $a=g+1-e>0$, then $2\Delta_{\infty}+a\Gamma$ is very ample, and a general member $D$ of the linear system $|2\Delta_{\infty}+a\Gamma|$ is a smooth irreducible hyperelliptic curve of genus $g$. Choose $D_0,D_1\in|2\Delta_{\infty}+a\Gamma|$, consider the Lefschetz pencil $\{\alpha D_0+\beta D_1\}$, $[\alpha\co \beta]\in\Pro^1$. Since $D_0\cdot D_1=(2\Delta_{\infty}+a\Gamma)^2= 4e+4a=4g+4.$ Therefore we have $4g+4$ distinct base points $p_1,\dots, p_{4g+4}$ of this pencil. Let $\nu\co S\to \Sigma_e$ be the blowing-up at these base points, then we obtain a fiberation $f\co S\to \Pro^1$ whose general fiber $F$ with genus $g$. Note that 
	\begin{align*}
		&F\sim\nu^*(2\Delta_{\infty}+a\Gamma)-E_1-\cdots -E_{4g+4},\\
		& K_S\sim\nu^*(-2\Delta_{\infty}+(e-2)\Gamma)+E_1 +\cdots +E_{4g+4},
	\end{align*}
	where $E_i$ is the exceptional divisor over $p_i$. Consider the linear system 
	\[|\nu^*(x\Delta_{\infty}+y\Gamma)-\sum_{i=1}^{4g+4}2E_i|.\]
	If \begin{equation}\label{eq: D}
		(x+1)(y+\frac{ex}{2})+x-3(4g+4)>0
	\end{equation}
	then it contains an irreducible member $D$, then we have $$K_S+D\sim \nu^*((x-2)\Delta_{\infty}+(y+e-2)\Gamma)-\sum_{i=1}^{4g-g}E_i.$$
	If \begin{equation}\label{eq: big}
		(x-2)(y+e-2+\frac{(x-2)e}{2})-\frac{1}{2}(4g+4)>0
	\end{equation}
then $K_S+D$ is big.
	Set $M=\nu^*((x-4)\Delta+(y+e-a-2)\Gamma)$, if 
	\begin{equation}\label{eq: effective}
		(x-3)(y+e-a-1+\frac{(x-4)e}{2})>0
	\end{equation} then we can assume $M\geq 0$. If
	\begin{equation}\label{eq: fixed part}
	(K_S+D)\cdot M=(x-4)(x-2)+(x-2)(y+e-a-2)+(x-4)(y+e-2)<0
	\end{equation}
then $M$ is contained in the fixed part of $|K_S+D|$. Note that $K_S+D-M\sim F$. So $|K_S+D|$ is composed of a pencil $|F|$. In this case the corresponding fiberation is $f\co X\to \Pro^1$, its general fiber has genus $g$ and $D\cdot F=x(g+1+e)+2y-8g-8$.

 We find that $x=8,y=1$, the inequalities (\ref{eq: D}), (\ref{eq: big}), (\ref{eq: effective}) and (\ref{eq: fixed part}) become
\begin{align*}
	&e>\frac{12g-5}{36} \quad (\ref{eq: D})
	&e>\frac{g+4}{12} \quad (\ref{eq: big})\\
	&e>\frac{g+1}{4} \quad (\ref{eq: effective})
	&e<\frac{3g-4}{8} \quad (\ref{eq: fixed part})
\end{align*}
When $g\geq 8$, we have $\frac{3g-4}{8}>\frac{12g-13}{36}>\frac{g+1}{4}>\frac{g+4}{12}$. Since $e$ is an integer, not every $g\geq 8$ the interval $(\frac{12g-13}{36},\frac{3g-4}{8})$ contains an integer, we find that $g=10,e=3$ satisfies all the inequalities, and when $g\geq 27$ we have $(\frac{12g-13}{36},\frac{3g-4}{8})\cap\Z\neq\varnothing$.  We can see that $D\cdot F=8e+2$, therefore the genus $g$ and the intersection number $k=D\cdot F$ can be arbitrarily large.

\end{example}

	\section*{Acknowledgements}
	 I would like to give heatly thanks my Ph.D supervisor JinXing Cai for suggesting generalizing Xiao's result about canonical maps of surface of log general type to log surfaces of log general type, and for his continuous support and encouragement, and to YiFei Chen, Lei Zhang, SongBo Ling, ZhiMin Ling and JingShan Chen for useful discussion. Many parts of this paper was accomplished during my Ph.D. The author was supported in part by the the science and technology innovation fund of Yunnan University: ST20210105.
	
	\bibliographystyle{alpha}
	
	\bibliography{ref}
	
\end{document}